\documentclass[11pt]{amsart}
\usepackage{amscd}
\usepackage{amsmath}
\usepackage{amsfonts}
\usepackage{amssymb}
\usepackage{mathrsfs}
\usepackage{graphicx}
\usepackage{color}
\usepackage{epstopdf}

\newtheorem{theorem}{Theorem}
\newtheorem{lemma}[theorem]{Lemma}
\newtheorem{proposition}[theorem]{Proposition}

\numberwithin{equation}{section}

\begin{document}

\title[Stationary Kirchhoff fractional critical problems]{Stationary Kirchhoff problems involving a fractional elliptic operator and a critical nonlinearity}

\author[G. Autuori]{Giuseppina Autuori}
\address{Dipartimento di Ingegneria Industriale e Scienze Matematiche, Universit\`{a} Politecnica delle Marche\\
Via Brecce Bianche, 12 -- 60131 Ancona, ITALY}
\email{\tt autuori@dipmat.unipvm.it}

\author[A.~Fiscella]{Alessio Fiscella}
\address{Dipartimento di Matematica ``Federigo Enriques'', Universit\`a di Milano\\
Via Cesare Saldini, 50 -- 20133 Milano, ITALY}
\email{\tt alessio.fiscella@unimi.it}
\address{Departamento de Matem\'atica, Universidade Estadual de Campinas, IMECC,
Rua S\'ergio Buarque de Holanda, 651, Campinas, SP CEP 13083--859 BRAZIL}
\email{\tt fiscella@ime.unicamp.br}

\author[P. Pucci]{Patrizia Pucci}
\address{Dipartimento di Matematica e Informatica, Universit\`{a} di Perugia\\
Via Vanvitelli, 1 -- 06123 Perugia, ITALY}
\email{\tt patrizia.pucci@unipg.it}


\begin{abstract}
This paper deals with the existence and the asymptotic behavior of non--negative solutions for a class of stationary Kirchhoff problems driven by a
fractional integro--differential ope\-ra\-tor $\mathcal L_K$ and involving a critical nonlinearity. In particular, we consider the problem
$$-M(\left\|u\right\|^{2})\mathcal L_Ku=\lambda f(x,u)+\left|u\right|^{{2^*_s} -2}u\quad \mbox{in }\Omega,
\qquad u=0\quad\mbox{in }\mathbb{R}^{n}\setminus\Omega,$$
where $\Omega\subset \mathbb R^n$ is a bounded domain, $2^*_s$ is the  critical exponent of the fractional Sobolev space $H^s(\mathbb{R}^n)$,
the function $f$ is a subcritical term and $\lambda$ is a positive parameter.
The main feature, as well as the main difficulty, of the analysis is the fact that the Kirchhoff function $M$ can be zero at zero,
that is  the problem is {\em degenerate}.
The adopted techniques are variational and the main theorems extend in several directions previous results recently appeared in the literature.
\end{abstract}

\maketitle

\section{Introduction}
In the last years, the attention towards nonlinear boundary value stationary Kirchhoff problems  has grown more and more,
thanks in particular to their intriguing analytical structure due to the presence of the nonlocal Kirchhoff function $M$
which makes the equation no longer a pointwise identity.
In the present paper we consider the problem
\begin{align}\label{P}
-&M\left(\|u\|^2\right)\mathcal L_Ku=\lambda f(x,u)\!+\!\left|u\right|^{2^*_s -2}\!u\quad\mbox{in } \Omega,\nonumber\\
&u=0\quad\mbox{in } \mathbb{R}^{n}\setminus\Omega,\\
&\|u\|^2=\iint_{\mathbb{R}^{2n}}\!\!\!\! |u(x)-u(y)|^2K(x-y)dxdy,\nonumber
\end{align}
where $\Omega\subset \mathbb R^n$ is a bounded domain, $n> 2s$,  with $s\in(0,1)$,  the number
\mbox{$2^*_s=2n/(n-2s)$} is the critical exponent of the fractional Sobolev space $H^s(\mathbb{R}^n)$,
the function $f$ is a sub\-cri\-tical term and $\lambda$ is a positive parameter.

The main nonlocal fractional operator $\mathcal L_K$ is defined for any $x\in\mathbb{R}^{n}$ by
\begin{equation}\label{nonloc}\mathcal L_K\varphi(x)=\frac{1}{2}\int_{\mathbb{R}^{n}}(\varphi(x+y)+\varphi(x-y)-2\varphi(x))K(y)dy,\end{equation}
along any $\varphi\in C^\infty_0(\Omega)$, where the kernel $K:\mathbb{R}^{n}\setminus \left\{0\right\}\to \mathbb R^+$ is a \textit{measurable} function
for which
\begin{enumerate}
\item[$(K_1)$]\label{K1} $m\,K\in L^{1}(\mathbb{R}^{n})$ \em{with} $m(x)=\min\left\{\left|x\right|^{2},1\right\}$;
\item[$(K_2)$]\label{K2} \it{there exists} $\theta >0$ \em{such that} $K(x)\geq\theta\left|x\right|^{-(n+2s)}$ \em{for any}  $x\in\mathbb{R}^{n}\setminus\{0\}$.
\end{enumerate}
A typical example of $K$ is given by $K(x)=\left|x\right|^{-(n+2s)}$. In this case the operator $\mathcal L_K=-(-\Delta)^{s}$
reduces to the fractional Laplacian, which (up to normalization factors) may be defined for any $x\in\mathbb{R}^{n}$ as
$$(-\Delta)^s \varphi(x)=\frac12\int_{\mathbb{R}^{n}}\frac{2\varphi(x)-\varphi(x+y)-\varphi(x-y)}{|y|^{n+2s}}dy$$
along any $\varphi\in C^\infty_0(\Omega)$; see \cite{valpal} and references therein for further details on the fractional
Laplace operator $(-\Delta)^s$ and the fractional Sobolev spaces $H^s(\mathbb{R}^n)$ and $H^s_0(\Omega)$.

Throughout the paper and without further mentioning, we assume that {\em the Kirchhoff function $M:\mathbb{R}^+_0 \to\mathbb{R}^+_0$ is  continuous} and,
since we are interested in non--negative solutions of~\eqref{P},  that {\em the nonlinearity  $f:\Omega\times\mathbb{R}\to\mathbb{R}$ is  a  Carath\'{e}odory function satisfying}
\begin{equation}\label{f4}
f=0\;\;\textit{in  $\Omega\times\mathbb R^-_0$\quad and\quad $f>0$ in }\Omega\times\mathbb R^+.
\end{equation}
A typical prototype for $M$, due to Kirchhoff, is given by
\begin{equation}\label{prot}M(t)=m_0 +b\gamma t^{\gamma-1},\quad m_0,b\ge0,\,\;m_0+b>0,\,\;\gamma\ge1.\end{equation}
When $M(t)>0$ for all $t\in\mathbb R^+_0$, Kirchhoff problems like \eqref{P} are said to be \textit{non--degenerate}
and this happens for example if $m_0>0$ in the model case \eqref{prot}.
While, if $M(0)=0$ but $M(t)>0$ for all $t\in\mathbb R^+$ Kirchhoff problems as~\eqref{P} are called
\textit{degenerate}. Of course, for \eqref{prot} this occurs when $m_0=0$.

The interest in studying problems like \eqref{P} relies not only on mathe\-ma\-ti\-cal purposes but also on their significance in real models, as explained by {\em Caffarelli} in \cite{caf1}--\cite{caf3}, and by {\em V\'azquez} in  \cite{v1}--\cite{v2}; see also~\cite{valpal}.

In the very recent paper \cite{FV}, {\it Fiscella} and {\it Valdinoci} provide a detailed discussion about the
physical meaning underlying the fractional Kirchhoff problems and their applications.
Indeed, {\it Fiscella} and {\it Valdinoci} propose in \cite{FV} a stationary Kirchhoff variational problem,  which models, as a special significant case, the {\em nonlocal} aspect of the tension arising from  nonlocal measurements of the fractional length of the string.
They consider problem \eqref{P}
for the first time in the literature, proving in~\cite[Theorem 1]{FV} the existence of a non--negative solution of \eqref{P}
for any $\lambda\ge \lambda^*>0$, where $\lambda^*$ is an appropriate threshold. They assume that $M$ is increasing in $\mathbb R^+_0$ and $M(0)>0$.

The first goal of this paper is to complete the picture in Theorem~\ref{Th1}, covering the {\it degenerate}
case $M(0)=0$, {\it without requiring any monotonicity assumption on} $M$. To this aim we suppose that
\begin{enumerate}
\item[$(M_1)$]\label{m1}
\textit{there exists} $\gamma\in [1,2^*_s/2)$ \textit{such that} $t M(t)\le \gamma\mathscr M(t)$ \textit{for any} $t\in\mathbb R_0^+$,
\textit{where} $\mathscr M(t)=\displaystyle\int_0^t M(\tau)d\tau$;
\item[$(M_2)$]\label{m2}
\textit{for any} $\tau>0$ \textit{there exists} $\kappa=\kappa(\tau)>0$ \textit{such that} $M(t)\geq \kappa$
\textit{for all} $t\geq \tau$;
\item[$(M_3)$]\label{m11} \textit{there exists} $a>0$ \textit{such that} $M(t)\ge a\,t$ \textit{for any} $t\in[0,1]$.
\end{enumerate}
It is worth noting that the restriction $\gamma\in [1,2^*_s/2)$, required in $(M_1)$, forces $n<2s\gamma/(\gamma-1)$ when $\gamma>1$.
In particular, the more interesting physical case in which $n=3$ and $\gamma=2$ is covered provided that $3/4<s<1$.

A very simple example of a degenerate Kirchhoff function $M$ satisfying $(M_1)$--$(M_3)$ is given by $M(t)=t$ if $t\in[0,1]$
and $M(t)=t^{m-1}$ if $t\ge1$, with $2\le m\le\gamma<2^*_s/2$. While the  prototype \eqref{prot} satisfies $(M_1)$--$(M_3)$,
provided that $\gamma\in[1,2^*_s/2)$ and either $a=b\,\gamma$ if $m_0=0$ and $\gamma\in[1,2]$, or $a=m_0$ if $m_0>0$.

We note in passing that $2<2^*_s/2$ occurs, provided that $n<4s$, a condition which we shall encounter for the {\it degenerate} case of \eqref{P}
in the main existence result of the paper, see Theorem~\ref{Th1}.

As noted in~\cite{colasuonno}, condition $(M_2)$ implies that $M(t)>0$ for any $t>0$, as in the prototype \eqref{prot}.
Hence, {\em in the general but special case $M(0)=m_0>0$}
condition $(M_2)$ yields at once that there exists $a>0$ such that
\begin{equation}\label{m30}
M(t)\ge a\quad\mbox{and}\quad\mathscr M(t)\geq at\quad\mbox{for any }t\in\mathbb R^+_0,
\end{equation}
so that in particular also $(M_3)$ trivially holds. A weaker condition of type \eqref{m30}, but covered in the degenerate case,
already appears in~\cite{ourr} for local Kirchhoff $p$--Laplacian critical Dirichlet problems,
where it is required that $\mathscr M(t)\ge C t^\gamma$ for all $t\in\mathbb R^+_0$ and an appropriate $\gamma\in(0,1]$.
More recently, in~\cite{molica2} assumption $(M_3)$ appears in the stronger form $M(t)\geq at$ for any $t\in\mathbb R^+_0$ and assuming also that $M$ is increasing in $\mathbb R^+_0$.

The nonlinearity $f$ is  related to the elliptic part by assuming that
\begin{enumerate}
\item[$(f_1)$]\label{f1}
$\lim\limits_{t\to 0^+} \dfrac{f(x,t)}{t^{2\gamma-1}}=0$, {\em uniformly in} $x\in\Omega$;
\item[$(f_2)$]\label{f2}
{\em there exists} $q\in (2\gamma,2^*_s)$ {\em such that} $\lim\limits_{t\to \infty} \dfrac{f(x,t)}{t^{q-1}}=0$,
{\em uniformly in} $x\in\Omega$;
\item[$(f_3)$]\label{f3}
{\em there exists} $\sigma\in (2\gamma,2^*_s)$ {\em such that} $\sigma F(x,t)\leq tf(x,t)$
{\em for all} $(x,t)\in\Omega\times\mathbb R^+$, {\em where} $F(x,t)=\displaystyle\int^{t}_{0}f(x,\tau)d\tau$.
\end{enumerate}
The function $f(x,t)=\phi(x)g(t)$, with $\phi\in L^\infty(\Omega)$, $\phi>0$ in $\Omega$,  $g(t)=0$ for $t\le0$  and
$g(t)=\sigma t^{\sigma-1}$ for $t>0$, verifies all the assumptions \eqref{f4}, $(f_1)$--$(f_3)$, provided that $2\gamma<\sigma<q<2^*_s$.

For the main result in the {\it degenerate} case we also require that $2s<n<4s$, which implies in the physical case $n=3$
and $\gamma=2$ that $3/4<s<1$. A restriction which already appears because of $(M_1)$.

\begin{theorem} \label{Th1}  Let $M(0)=0$ and $2s<n<4s$. Assume that $K$, $M$ and $f$ satisfy \eqref{f4}, $(K_1)$--$(K_2)$,
$(M_1)$--$(M_3)$ and $(f_1)$--$(f_3)$. Then there exists $\lambda^*>0$ such that for any $\lambda\ge\lambda^*$
problem \eqref{P} admits a non--trivial non--negative mountain pass solution $u_\lambda$. Moreover
\begin{equation}\label{asym}
\lim_{\lambda\to\infty}\left\|u_\lambda\right\|=0.
\end{equation}
\end{theorem}
The symbol $\|\cdot\|$ in Theorem~\ref{Th1} denotes the norm of the solution and test function space $Z$,
which will be introduced in Section~\ref{sec variational}. Of course the restriction $2s<n<4s$ forces to cover only the case
$n\in\{1,2,3\}$.
In \cite{molica2} a problem of type \eqref{P} is treated, when the right hand side  is replaced by two parametric subcritical nonlinearities, assuming
$M$ increasing in $\mathbb R^+_0$, condition  $(M_3)$ in the stronger form $M(t)\ge a\,t$ for any $t\in\mathbb R^+_0$ and that there exists $m>1$ such that
$\liminf_{t\to\infty}t^{-m}\mathscr M(t)>0$. The main theorems of \cite{molica2} show that if $2s<n<4s$ and the subcritical nonlinearities satisfy some technical conditions, the corresponding Dirichlet problem in bounded domains, with Lipschitz boundary,  admits at least three solutions with Gagliardo norms sufficiently small.

In order to prove Theorem \ref{Th1} we somehow adapt the variational approach used by {\it Fiscella} and {\it Valdinoci}
in the proof of \cite[Theorem~1]{FV} and by {\it Colasuonno} and {\it Pucci} in the proof of
\cite[Lemma 3.1]{colasuonno}, but using complectly new and simpler arguments. Clearly, the truncation technique on $M$ used  in~\cite{Fig,FV} is not worth in
the {\it degenerate} case.
Indeed, by $(M_2)$ we still have positive boundedness from below for $M$ outside a neighborhood of zero, but it is not uniform anymore as in \cite{Fig,FV}.
Hence, some difficulties arise in trying to balance the boundedness from above given by truncation, which is an a priori estimate.
For this, assumptions $(M_1)$--$(M_3)$ will play a crucial role.

In the second part of the paper, we extend~\cite[Theorem 1]{FV} proving the existence of a non--trivial non--negative solution of \eqref{P}
for any $\lambda$ sufficiently large, in the {\it non--degenerate} case, under suitable assumptions for $M$ and $f$.
More precisely, on the Kirchhoff function $M$, {\em we require only that}
\begin{equation}\label{m'2}
\inf_{t\in\mathbb R^+_0}M(t)=a>0;
\end{equation}
in other words we demand  that \eqref{m30} holds. Of course \eqref{m'2} implies also the validity of~$(M_2)$ and~$(M_3)$.
The nonlinearity $f$ is assumed to verify
\begin{enumerate}
\item[$(F_1)$]\label{f'1}
$\lim\limits_{t\to 0} \dfrac{f(x,t)}{t}=0$, {\em uniformly in} $x\in\Omega$;
\item[$(F_2)$]\label{f'2}
{\em there exists} $q\in (2,2^*_s)$ {\em such that} $\lim\limits_{t\to \infty} \dfrac{f(x,t)}{t^{q-1}}=0$,
{\em uniformly in} $x\in\Omega$;
\item[$(F_3)$]\label{f'3}
{\em there exists}  $\sigma\in (2,2^*_s)$ {\em such that} $\sigma F(x,t)\leq tf(x,t)$ {\em for any}
$(x,t)\in\Omega\times\mathbb R^+$.
\end{enumerate}
Note that $(F_1)$--$(F_3)$ coincide with the assumptions $(8)$--$(10)$ required in~\cite{FV}.

In the {\it non--degenerate} setting, the prototype \eqref{prot} satisfies \eqref{m'2} with any exponent \mbox{$\gamma\ge1$,} with no upper bounds.
However, in order to give an extension of~\cite[Theorem 1]{FV} and in the meanwhile treating the almost complete {\em non--degenerate} situation we consider and obtain
two different results, stated just below.

\begin{theorem} \label{Th2} Let $K$, $M$ and $f$ satisfy \eqref{f4}, \eqref{m'2}, $(K_1)$--$(K_2)$ and $(F_1)$--$(F_3)$.\\
$\phantom{ii}(i)$ If  there exists $\gamma\in [1,\sigma/2)$ such that  $t M(t)\le \gamma\mathscr M(t)$  for all $t\in\mathbb R_0^+$, 
then there exists $\lambda^*>0$ such that for any $\lambda\ge\lambda^*$ problem~\eqref{P} admits a non--trivial non--negative mountain pass solution $u_\lambda$.\\
$\phantom{i}(ii)$ If $2M(0)<\sigma a$, then for all $\alpha\in(M(0),\sigma a/2)$ there exists \mbox{$\lambda^*=\lambda^*(\alpha)>0$} such that for any $\lambda\ge\lambda^*$ problem~\eqref{P} has a non--trivial non--negative mountain pass solution $u_\lambda$. 

In both cases, the mountain pass solution $u_\lambda$ satisfies property~\eqref{asym}.
\end{theorem}

Clearly the request $2M(0)<\sigma a$ is automatic whenever $M(0)=a$, being $\sigma>2$ by $(F_3)$. The assumption $M(0)=a$, together with monotonicity of $M$, was assumed in \cite{FV}. A very interesting open problem is to construct a non--trivial solution $u_\lambda$ for \eqref{P} when $\sigma a\le2M(0)$ and the growth condition on $M$ stated in $(i)$ does not hold.

As we shall see, the  proof of Theorem~$\ref{Th2}$--$(i)$ is performed just adapting the approach used in Theorem~\ref{Th1}
and, as in the {\it degenerate} case, we get existence for any $\lambda$ beyond a threshold $\lambda^*>0$ and again extend \cite[Theorem 1]{FV} in several directions. Theorem~\ref{Th2}--$(ii)$ is proved via a truncation argument on $M$, since the Kirchhoff function $M$ could increase too quickly with respect
the other terms of problem \eqref{P}. This argument was already used in~\cite{FV}, as well as in~\cite{Fig}. With Theorem~\ref{Th2} we extend \cite[Theorem 1]{FV} in several directions.
First of all, we  do not require any longer $M(0)=a$, and even in Theorem~\ref{Th2}--$(ii)$ we only assume the milder restriction $2M(0)<\sigma a$, which is
automatic when $M(0)=a$. 

It is worth mentioning that for the Kirchhoff function 
$$M(t)=(1+t)^m+(1+t)^{-1},\quad t\in\mathbb R^+_0,\quad m\in(0,1),$$
we have $M(0)=2$ and $a=m^{-m/(m+1)}(1+m)<2$. Furthermore,  if  $m$ is so small that $m+1<\sigma/2$, then 
$M$ verifies all the assumptions of Theorem~$\ref{Th2}$--$(i)$, with $\gamma=m+1$. While, if $m\in(0,1)$ is sufficiently large, then $2\,M(0)=4<\sigma\,a$, being $2<\sigma$ by $(F_3)$. Thus,
$M$ verifies all the assumptions of Theorem~$\ref{Th2}$--$(ii)$.
It is therefore evident that Theorem~$\ref{Th2}$ is applicable even when neither $M$
is increasing in $\mathbb R^+_0$, nor $M(0)=a$.

Last but not least, throughout the paper, we assume on $K$ only $(K_1)$
and $(K_2)$, without asking that $K$ is controlled from above by the fractional kernel $\left|x\right|^{-(n+2s)}$, as  in~\cite{FV}.

Several recent papers are focused both on theoretical aspects and applications related to nonlocal fractional models.
Concerning the critical case, in \cite{colorado} the effects of lower order perturbations are studied
for the existence of positive solutions of critical elliptic problems involving the fractional Laplacian.
In \cite{colorado2}, see also the references therein, existence and multiplicity results are established
for a critical fractional equation with concave--convex  nonlinearities.
In \cite{sv3} a Brezis--Nirenberg existence result for nonlocal fractional equations is proved through variational methods,
while in \cite{sv2} the authors extend the theorems got in \cite{sv3} to a more general problem, again involving an integro--differential
nonlocal operator and critical terms. Furthermore, a multiplicity result for a Brezis--Nirenberg problem in nonlocal fractional setting
is given in \cite{fms}, where it is shown that in a suitable left neighborhood of any eigenvalue of $-\mathcal L_K$ (with Dirichlet boundary data)
the number of solutions is at least twice the multiplicity of the eigenvalue.
For multiplicity results on a {\it non--degenerate} stationary Kirchhoff problem involving $-\mathcal L_K$ and a nonlinearity of integral
form we refer to~\cite{mr} and the references therein.

For evolutionary Kirchhoff problems it is worth mentioning the paper \cite{acp}, which is concerned with lifespan estimates of maximal solutions
for {\it degenerate} polyharmonic Kirchhoff problems. The technique goes back to \cite{pucciserrin} and it is based on the construction of a
Lyapunov function $\mathscr Z$ which lives as longer as any local solution $u$ does. The goal is to show that $\mathscr Z$
becomes unbounded in finite time, proving the non--continuation of $u$.
A priori estimates for the maximal living time $T$ are obtained exploiting in a suitable way the non--existence tools already adopted in
\cite{aps} for (possibly {\it degenerate}) $p(x)$--Kirchhoff systems involving nonlinear damping and source terms.
More recently, in \cite{colasuonno} a multiplicity result is obtained for a {\it degenerate} stationary
Kirchhoff problem governed by the $p(x)$--polyharmonic operator, with a subcritical term, through the Mountain Pass theorem,
while in \cite{acp2} existence and multiplicity of solutions of certain eigenvalue stationary $p$--polyharmonic Kirchhoff problems are
considered, also in a {\it degenerate} setting.

Furthermore, the very interesting paper \cite{puccisal} treats the question of the existence and multiplicity of nontrivial non--negative entire solutions
of a Kirchhoff eigenvalue problem, involving critical nonlinearities and nonlocal elliptic operators, but only in the {\it non--degenerate} setting.
We also refer to~\cite{Fig2} and the references therein for the large literature on Kirchhoff--Schr\"odinger type equations in the {\it non--degenerate} case.

Finally, paper \cite{ourr} deserves a special mention for our purposes.
In our knowledge, it is the only work facing {\it degenerate} Kirchhoff problems involving a critical nonlinearity.
To overcome the lack of compactness at critical level the author uses a compactness result which exploits the pointwise
convergence of the gradients of a sequence of solutions.  As we shall see,
we do not have derivatives of solutions in $Z$, but a sort of integro--differentiation, cf. \eqref{normaz}.
Thus, using the celebrated {\it Brezis} and {\it Lieb} lemma, as explained in the proof of Lemma~\ref{palais},
an adaptation of the variational approach of~\cite{ourr} could be possible in the fractional framework of this paper, but we prefer to use
a new approach.

Inspired by the above papers and motivated by the poor literature in the {\it degenerate} critical case,
we provide existence and asymptotic results for \eqref{P}.

The paper is organized as follows. In Section~\ref{sec variational} we discuss the variational formulation of the problem and show that solutions of~\eqref{P} are non--negative.
In Sections~\ref{sec degenerate} and \ref{sec non-degenerate}  we prove
Theorems~\ref{Th1} and~\ref{Th2}, respectively.\section{Variational formulation}\label{sec variational}

Problem \eqref{P} has a variational structure and the natural space where finding solutions is the homogeneous fractional Sobolev space $H^s_0(\Omega)$,
see \cite{valpal}. In order to study \eqref{P} it is important to encode the ``boundary condition'' $u=0$ in
$\mathcal C\Omega=\mathbb{R}^n\setminus\Omega$ (which is different from the classical case of the Laplacian, where it is required $u=0$ on $\partial \Omega$) in the weak formulation,
by considering also that the interaction between $\Omega$ and its complementary ${\mathcal C}\Omega$ in $\mathbb{R}^n$ gives a positive contribution
in the norm $\left\|u\right\|_{H^s(\mathbb{R}^n)}$. This is exactly the nonlocal nature of the elliptic operator, as it is transparent  from \eqref{nonloc}. The functional space that takes into account this boundary condition will be denoted by $Z$
and it was introduced in \cite{fiscella} in the following way.

First, let $X$ be the linear space of Lebesgue measurable functions \mbox{$u:\mathbb{R}^n\to\mathbb{R}$} whose restrictions to $\Omega$ belong to $L^2(\Omega)$ and such that
	\[\mbox{the map }
(x,y)\mapsto (u(x)-u(y))^2 K(x-y) \mbox{ is in } L^1\big(Q,dxdy\big),
\]
where $Q=\mathbb{R}^{2n}\setminus\left({\mathcal C}\Omega\times{\mathcal C}\Omega\right)$.
The space $X$ is endowed with the norm
\begin{equation}\label{norma}
\left\|u\right\|_{X}=\Big(\|u\|_{L^2(\Omega)}+\iint_Q |u(x)-u(y)|^2K(x-y)dxdy\Big)^{1/2}.
\end{equation}
It is easy to see that bounded and Lipschitz functions belong to $X$, thus $X$ is not reduced to
$\left\{0\right\}$; see \cite{sv3,sv2} for further details on the space $X$.

The functional space $Z$ denotes the closure of $C^{\infty}_{0}(\Omega)$ in $X$.
The scalar product defined for any $\varphi$, $\phi\in Z$ as
\begin{equation}\label{scal}
\langle \varphi, \phi\rangle_Z=\iint_{Q} (\varphi(x)-\varphi(y))(\phi(x)-\phi(y))K(x-y) dxdy,
\end{equation}
makes $Z$ a Hilbert space. The norm
\begin{equation}\label{normaz}
\|u\|_{Z}=\Big(\iint_Q |u(x)-u(y)|^2K(x-y)dxdy\Big)^{1/2}
\end{equation}
is equivalent to the usual one defined in \eqref{norma}, as proved in \cite[Lemma~4]{fiscella}.
Note that in \eqref{norma}--\eqref{normaz} the integrals can be extended to all $\mathbb{R}^{n}$ and $\mathbb{R}^{2n}$, since  $u=0$ a.e.
in $\mathcal C\Omega$. {\it From now on, in order to simplify the notation, we denote $\langle \cdot, \cdot\rangle_Z$ and
$\|\cdot\|_{Z}$ by $\langle \cdot, \cdot\rangle$ and $\|\cdot\|$, respectively.}

The Hilbert space $Z=(Z, \|\cdot\|)$ is continuously embedded in $H^s_0(\Omega)$ and compactly embedded
in $L^r(\Omega)$ for any $r\in[1, 2^*_s)$, see~\cite[Lemma 4]{fiscella}.

Even if $Z$ is not a real space of functions but a density space, the choice of this functional space is an improvement with respect of the usual space
$X_0=\left\{u\in X:u=0\;\,\mbox{a.e. in }\mathbb{R}^n\setminus
\Omega\right\}$, where many authors set their nonlocal variational problems, e.g. \cite{colorado}--\cite{capella,molica,sv1}--\cite{sv2}.
Indeed, the density result proved in \cite[Theorem 6]{fsv} does not hold true without more restrictive conditions on the boundary $\partial\Omega$;
see in particular \cite[Remark 7]{fsv}. Thus, when $\Omega$ is simply a bounded domain of $\mathbb{R}^n$, we just have $Z\subset X_0$, with possibly $Z\not= X_0$.

The weak formulation of \eqref{P} is as follows. We say that $u\in Z$ is a (weak) {\em solution} of~\eqref{P}~if
\begin{equation}\label{wf}M\left(\|u\|^2\right)\langle u, \varphi\rangle=
\lambda\displaystyle{\int_\Omega f(x, u(x))\varphi(x)dx}+\int_\Omega \left|u(x)\right|^{2^*_s-2}u(x)\varphi(x)dx\end{equation}
for all $\varphi \in Z$.

The structural assumptions on $\Omega$, $M$, $f$ and $K$ assure that all the integrals in \eqref{wf} are well defined if $u$, $\varphi\in Z$.
Moreover, it is worth to say that the odd part of $K$ does not give contribution in the integral of the left--hand side of \eqref{wf}.
Indeed, write $K=K_e+K_o$, where for all $x\in\mathbb R^n\setminus\{0\}$
$$K_e(x)=\frac{K(x)+K(-x)}{2}\quad\mbox{and}\quad K_o(x)=\frac{K(x)-K(-x)}{2}.$$
Then, it is apparent that for all $u$ and $\varphi \in Z$
$$\langle u, \varphi\rangle=\iint_{\mathbb{R}^{2n}} (u(x)-u(y))(\varphi(x)-\varphi(y))K_e(x-y) dxdy.$$
Therefore, {\em it is not restrictive to assume $K$ even}.

According to the variational nature, (weak) solutions of \eqref{P} correspond to critical points of the associated
Euler--Lagrange functional $\mathcal J_\lambda:Z\to \mathbb{R}$ defined by
$$\mathcal J_\lambda(u)=\frac{1}{2}\mathscr{M}(\left\|u\right\|^2)-\lambda\int_\Omega F(x, u(x))dx
-\frac{1}{2^*_s}\|u\|^{2^*_s}_{2^*_s}$$
for all $u\in Z$. Note that $\mathcal J_\lambda$ is a $C^1(Z)$ functional  for any $u \in Z$
\begin{equation}\label{derivata}
\mathcal J'_\lambda(u)(\varphi)= M(\left\|u\right\|^2)\langle u, \varphi\rangle-\lambda\int_\Omega f(x, u(x))\varphi(x)dx-\int_\Omega \left|u(x)\right|^{2^*_s-2}u(x)\varphi(x)dx
\end{equation}
for all $\varphi \in Z$.

We recall that throughout the paper we assume \eqref{f4} and $(K_1)$--$(K_2)$.

\begin{proposition}\label{prop} For all $\lambda\in\mathbb R$ any solution of \eqref{P} is non--negative.
\end{proposition}

\begin{proof} Let $u_\lambda$ be a solution of problem \eqref{P}. By \cite[Lemma 8]{FV} we have $u^-_\lambda\in Z$. Thus, by \eqref{wf}
with $\varphi=u^-_\lambda$ we get
\begin{equation}\label{u-}M(\left\|u_\lambda\right\|^2)\langle u_\lambda, u^-_\lambda\rangle=
\lambda\int_\Omega f(x, u_\lambda(x))u^-_\lambda(x)dx+\|u^-_\lambda\|^{2^*_s}_{2^*_s}.\end{equation}
We observe that for a.a. $x$, $y\in\mathbb{R}^n$
$$\begin{aligned}
(u_\lambda(x)&-u_\lambda(y))(u^-_\lambda(x)-u^-_\lambda(y))\\
&=-u^+_\lambda(x)u^-_\lambda(y)-u^-_\lambda(x)u^+_\lambda(y)-\big[u^-_\lambda(x)-u^-_\lambda(y)\big]^2\\
&\leq-\left|u^-_\lambda(x)-u^-_\lambda(y)\right|^2.
\end{aligned}$$
Moreover, $f(x, u_\lambda(x))u^-_\lambda(x)=0$ for a.a. $x\in\mathbb{R}^n$ by \eqref{f4}.
Hence, by \eqref{u-} it follows that
$$0\ge -M(\left\|u_\lambda\right\|^2)\iint_{\mathbb{R}^{2n}} \left|u^-_\lambda(x)-u^-_\lambda(y)\right|^2 K(x-y) dxdy\ge\|u^-_\lambda\|^{2^*_s}_{2^*_s},$$
being $M\ge0$. Thus $u^-_\lambda\equiv0$, that is $u_\lambda$ is non--negative, as required.
\end{proof}


\section{The degenerate case}\label{sec degenerate}

In this section we discuss the main novelty of the paper, that is the existence of non--trivial solutions of \eqref{P} in
the {\it degenerate} case $M(0)=0$. {\it Throughout this section we assume \eqref{f4}, $(K_1)$--$(K_2)$, $(M_1)$--$(M_2)$ and $(f_1)$--$(f_3)$,} without further mentioning.

In order to find the critical points of $\mathcal J_\lambda$, we intend to apply the Mountain Pass theorem of {\em Ambrosetti} and {\em Rabinowitz} \cite{ar}.
For this, we have to check that $\mathcal J_\lambda$ possesses a suitable geometrical structure and that it satisfies the Palais--Smale compactness condition.
To prove all these properties, we need appropriate lower and upper bounds for $M$, $f$ and their primitives.

As noted in \cite{colasuonno}, condition $(M_2)$ implies that $M(t)>0$ for any $t>0$ and consequently, using $(M_1)$, we get
\begin{equation}\label{m3}
\mathscr M(t)\geq \mathscr M(1)t^\gamma\quad\mbox{for any }t\in[0,1].
\end{equation}
Of course also condition $(M_3)$ implies that $\mathscr M(t)\geq a\,t^2/2$ for all $t\in[0,1]$ and this together with \eqref{m3} gives  $\mathscr M(t)\geq c\,t^m$ for all $t\in[0,1]$,
where $m=\min\{2,\,\gamma\}$.
Similarly, for any $\varepsilon>0$ there exists $r_\varepsilon=r(\varepsilon)=\mathscr M(\varepsilon)/\varepsilon^\gamma >0$ such that
\begin{equation}\label{m4}
\mathscr M(t)\leq r_\varepsilon t^\gamma\quad\mbox{for any }t\geq\varepsilon.
\end{equation}
Hence, property \eqref{m4} yields that
\begin{equation}\label{m5}
\lim_{t\to\infty}\frac{\mathscr M(t)}{t^{\sigma/2}}=0,
\end{equation}
being $\sigma>2\gamma$ by $(f_3)$.

Concerning the nonlinear term $f$, assumptions $(f_1)$ and $(f_2)$ give subcritical growths.
That is, for any $\varepsilon>0$ there exists $\delta_\varepsilon=\delta(\varepsilon)>0$ such that
\begin{equation}\label{f11}
0\le f(x,t)\leq 2\gamma\varepsilon\left|t\right|^{2\gamma-1} +q\delta_\varepsilon\left|t\right|^{q-1}\quad\mbox{for any }(x,t)\in\Omega\times\mathbb{R}
\end{equation}
by \eqref{f4} and so for the primitive
\begin{equation}\label{f12}
0\le F(x,t)\leq \varepsilon\left|t\right|^{2\gamma} +\delta_\varepsilon\left|t\right|^q\quad\mbox{for any }(x,t)\in\Omega\times\mathbb{R}.
\end{equation}
Finally,  $(f_3)$ implies that $F(x,t)\geq c(x) t^\sigma$ for all $(x,t)\in\Omega\times[1,\infty)$,
where $c(x)=F(x,1)$ is in $L^\infty(\Omega)$ by \eqref{f12}, with $\varepsilon=t=1$. In conclusion, for any $(x,t)\in\Omega\times\mathbb{R^+}$
\begin{equation}\label{f31}
F(x,t)\geq c(x) t^\sigma-C(x),\quad C(x)=\max_{t\in[0,1]}\big|F(x,t)-c(x) t^\sigma\big|.
\end{equation}
Again $C\in L^\infty(\Omega)$  by \eqref{f12}.\smallskip

In the next two lemmas we prove that for any $\lambda>0$ the functional $\mathcal J_\lambda$ satisfies all
the geometric features required by the Mountain Pass theorem. Later, in the crucial and most delicate Lemma~\ref{palais} we show that
the functional $\mathcal J_\lambda$ satisfies the Palais--Smale condition at a suitable level $c_\lambda$ beyond a threshold $\lambda^*>0$.

\begin{lemma}\label{mp1} For any $\lambda>0$ there exist two positive constants $\alpha$ and $\rho$ such that
$\mathcal J_\lambda(u)\geq\alpha>0$ for any $u\in Z$, with $\left\|u\right\|=\rho$.
\end{lemma}
\begin{proof} Fix $\lambda>0$ and take $u\in Z$, with $\|u\|\le1$. Let $\varepsilon>0$ be sufficiently small to be chosen later. By \eqref{m3} and \eqref{f12}
$$\mathcal J_\lambda(u)\geq \frac{\mathscr M(1)}{2}\left\|u\right\|^{2\gamma}-\varepsilon\lambda\|u\|^{2\gamma}_{2\gamma}
-\delta_\varepsilon\lambda\|u\|^{q}_{q}-\frac{1}{2^*_s}\|u\|^{2^*_s}_{2^*_s}.$$
The fractional Sobolev inequality proved in \cite[Theorem 6.5]{valpal} assures the exi\-sten\-ce of positive constants
$S_{2\gamma}$, $S_q$ and $S_{2^*_s}$  for which
$$\mathcal J_\lambda(u)\geq \left(\frac{\mathscr M(1)}{2}-\varepsilon\lambda S_{2\gamma}\right)\left\|u\right\|^{2\gamma}-\delta_\varepsilon
\lambda S_q\left\|u\right\|^{q}-S_{2^*_s}\left\|u\right\|^{2^*_s}.$$
Therefore, fixing $\varepsilon>0$ so small that $\mathscr M(1)>2\varepsilon\lambda S_{2\gamma}$ and consequently cho\-osing $\rho$  small enough,
we obtain the result, thanks to the fact that $2\gamma<q<2^*_s$ by $(f_2)$. 
\end{proof}

\begin{lemma}\label{mp2} For any $\lambda>0$ there exists $e\in Z$, with $e\ge0$ a.e. in $\mathbb{R}^n$, \mbox{$\mathcal J_\lambda(e)<0$} and $\left\|e\right\|>\rho$,
where $\rho$ is given in Lemma~$\ref{mp1}$.
\end{lemma}
\begin{proof} Fix $\lambda>0$ and take $u_0\in Z$ such that $u_0\geq 0$ a.e. in $\mathbb{R}^n$ and $\left\|u_0\right\|=1$. By \eqref{m5} there exists
$t_0=t_0(2)>0$ such that $\mathscr M(t^2)\leq2t^{\sigma}$ for any $t\geq t_0$.
Therefore, by using \eqref{f31} we find
$$\mathcal J_\lambda(tu_0)\leq t^\sigma\left(1-\lambda\int_\Omega c(x)\left|u_0(x)\right|^\sigma dx\right)
-\frac{t^{2^*_s}}{2^*_s}\|u_0\|^{2^*_s}_{2^*_s}+\lambda\|C\|_1,$$ for any $t\geq t_0$.
Since $\sigma<2^*_s$, passing to the limit as $t\to \infty$, we obtain that $\mathcal J_\lambda(tu_0)\to -\infty$.
Hence the assertion follows by taking $e=t_{*}u_0$, with $t_* >0$ large enough. 
\end{proof}

We discuss now the compactness property for the fun\-ctio\-nal~$\mathcal J_\lambda$, given by the Palais--Smale condition at a suitable level.
For this, we fix $\lambda>0$ and set
\begin{equation}\label{mm}c_\lambda=\inf_{g\in\Gamma}\max_{t\in[0,1]}\mathcal J_\lambda(g(t)),\end{equation}
where
$$\Gamma=\left\{g\in C([0,1],Z):g(0)=0,\;\mathcal J_\lambda(g(1))<0\right\}.$$
Clearly, $c_\lambda>0$ by Lemma~\ref{mp1}.
We recall that $\left\{u_{j}\right\}_{j\in\mathbb{N}}\subset Z$ is a {\it Palais--Smale sequence for $\mathcal J_\lambda$ at level $c_\lambda\in\mathbb{R}$} if
\begin{equation}\label{ps1}
\mathcal J_\lambda(u_{j})\to c_\lambda\quad\mbox{and}\quad\mathcal J'_\lambda(u_{j})\to0\quad\mbox{as}\;j\to \infty.
\end{equation}
We say that  $J_\lambda$ {\it satisfies the Palais--Smale condition at level} $c_\lambda$ if any Palais--Smale sequence
$\left\{u_{j}\right\}_{j\in\mathbb{N}}$ at level $c_\lambda$ admits a convergent subsequence in $Z$.

Before proving the relatively compactness of the Palais--Smale sequences, we introduce an asymptotic condition for the level $c_\lambda$.
This result will be crucial not only to get \eqref{asym}, but above all to overcome the lack of compactness due to the presence of a critical nonlinearity.
\begin{lemma}\label{infinito} It results
$$\lim_{\lambda\to\infty}c_\lambda=0.$$
\end{lemma}
\begin{proof} Fix $\lambda>0$ and let $e\in Z$ be the function obtained by Lemma \ref{mp2}.
Since $\mathcal J_\lambda$ satisfies the Mountain Pass geometry, there exists $t_\lambda>0$ verifying
$\mathcal J_\lambda(t_\lambda e)=\displaystyle\max_{t\geq 0}J_\lambda(te)$. Hence,
$\mathcal J'_\lambda(t_\lambda e)(e)=0$ and by \eqref{derivata}
\begin{equation}\label{3.1}
t_\lambda \left\|e\right\|^2 M(t^{2}_{\lambda}\left\|e\right\|^2)=
\lambda\int_\Omega f(x, t_\lambda e(x))e(x)dx+t^{2^*_s-1}_{\lambda}\|e\|^{2^*_s}_{2^*_s}\ge t^{2^*_s-1}_{\lambda}\|e\|^{2^*_s}_{2^*_s}
\end{equation}
by \eqref{f4} and the fact that $\lambda>0$. We claim that $\{t_\lambda\}_{\lambda>0}$ is bounded. Indeed, fix $\varepsilon>0$. By \eqref{m5} there exists $t_0=t_0(\varepsilon)>0$ such that
$\mathscr M(t)\le\varepsilon t^{\sigma/2}$ for any $t\ge t_0$. Thus, denoting by $\Lambda=\{\lambda>0:\; t^{2}_{\lambda}\left\|e\right\|^2\geq t_0\}$,
we see that
\begin{equation}\label{3.2}
t^2_\lambda \left\|e\right\|^2 M(t^{2}_{\lambda}\left\|e\right\|^2)\le
\gamma\mathscr M(t^{2}_{\lambda}\left\|e\right\|^2)\le\varepsilon  \gamma t^\sigma_\lambda \left\|e\right\|^\sigma \quad \mbox{for any }\lambda\in\Lambda
\end{equation}
by $(M_1)$. Since by construction $e\ge0$ a.e. in $\mathbb{R}^n$ and $\left\|e\right\|>\rho$, from \eqref{3.1} and \eqref{3.2} it follows
$$\varepsilon  \gamma\left\|e\right\|^\sigma\geq t^{2^*_s-\sigma}_{\lambda}\|e\|^{2^*_s}_{2^*_s} \quad \mbox{for any }\lambda\in\Lambda,$$
which implies the boundedness of $\{t_\lambda\}_{\lambda\in\Lambda}$.
Clearly by construction of $\Lambda$ also $\{t_\lambda\}_{\lambda\in(\mathbb R\setminus\Lambda)}$ is bounded.
This concludes the proof of the claim.

Fix now a sequence $\{\lambda_j\}_{j\in\mathbb{N}}\subset \mathbb R^+$ such that $\lambda_j\to\infty$ as $j\to\infty$. Clearly
$\{t_{\lambda_j}\}_{j\in\mathbb{N}}$ is bounded. Hence, there exist a subsequence of $\{\lambda_j\}_{j\in\mathbb{N}}$, that we still denote by $\{\lambda_j\}_{j\in\mathbb{N}}$, and a constant
$t_0\geq 0$ such that $t_{\lambda_j}\to t_0$ as $j\to\infty$.
By continuity of $M$, also $\left\{M(t^{2}_{\lambda_j}\left\|e\right\|^2)\right\}_{j\in\mathbb{N}}$ is bounded, and so by \eqref{3.1}
there exists $D>0$ such that for any $j\in\mathbb{N}$
\begin{equation}\label{D}
\lambda_j\int_\Omega f(x, t_{\lambda_j} e(x))e(x)dx+t^{2^*_s-1}_{\lambda_j}\|e\|^{2^*_s}_{2^*_s}\leq D.
\end{equation}
We assert that $t_0=0$. Indeed, if $t_0>0$ then by \eqref{f11} and the Dominated Convergence Theorem
$$\int_\Omega f(x, t_{\lambda_j} e(x))e(x)dx\to\int_\Omega f(x, t_0 e(x))e(x)dx>0\quad\mbox{as }j\to\infty$$
by \eqref{f4}. Recalling that $\lambda_j\to\infty$, we get
$$\lim_{j\to\infty}\left(\lambda_j\int_\Omega f(x, t_{\lambda_j} e(x))e(x)dx+t^{2^*_s-1}_{\lambda_j}\|e\|^{2^*_s}_{2^*_s}\right)=\infty,$$
which contradicts \eqref{D}. Thus $t_0=0$ and $t_\lambda\to0$ as $\lambda\to\infty$, since the sequence $\{\lambda_j\}_{j\in\mathbb{N}}$ is arbitrary.

Consider now the path $g(t)=te$,  $t\in[0,1]$, belonging to $\Gamma$. By Lemma~\ref{mp1} and \eqref{f4}
$$0<c_\lambda\leq\max_{t\in[0,1]}\mathcal J_\lambda(g(t))\leq\mathcal J_\lambda(t_\lambda e)
\leq\frac{1}{2}\mathscr M(t^{2}_{\lambda}\left\|e\right\|^2),$$
where  $\mathscr M(t^{2}_{\lambda}\left\|e\right\|^2)\to 0$ as $\lambda\to\infty$ by continuity. This completes the proof of the lemma.
\end{proof}

Now, we are ready to prove the Palais--Smale condition.

\begin{lemma}\label{palais} There exists $\lambda^*>0$ such that for any $\lambda\ge\lambda^*$ the functional $J_\lambda$ satisfies the Palais--Smale condition at level $c_\lambda$.
\end{lemma}

\begin{proof} Take $\lambda>0$ and let $\left\{u_{j}\right\}_{j\in\mathbb{N}}\subset Z$ be a Palais--Smale sequence for $\mathcal J_\lambda$
at level $c_\lambda$. Due to the degenerate nature of \eqref{P}, two situations must be considered:
either \mbox{$\displaystyle\inf_{j\in\mathbb{N}}\left\|u_j\right\|=d_\lambda>0$} or $\displaystyle\inf_{j\in\mathbb{N}}\left\|u_j\right\|=0$.
For this, we divide the proof in two cases. \medskip

\noindent {\it Case $\displaystyle\inf_{j\in\mathbb{N}}\left\|u_j\right\|=d_\lambda>0$.}
First we prove that $\left\{u_{j}\right\}_{j\in\mathbb{N}}$ is bounded in $Z$.
By $(M_2)$ with $\tau=d_\lambda^2$ there exists $\kappa_\lambda>0$ such that
\begin{equation}\label{mk}
M(\left\|u_j\right\|^2)\geq \kappa_\lambda\quad\mbox{for any }j\in\mathbb{N}.
\end{equation}
Furthermore, from $(M_1)$ and $(f_3)$ it follows that
\begin{equation}\label{4.2}
\begin{aligned}
\mathcal J_\lambda(u_j)&-\frac1\sigma\mathcal J'_\lambda(u_j)(u_j)\\
&\ge\frac{1}{2}\mathscr M(\left\|u_j\right\|^2)-\frac{1}{\sigma}M(\left\|u_j\right\|^2)\left\|u_j\right\|^2+\left(\frac{1}{\sigma}
-\frac{1}{2^*_s}\right)\|u_j\|^{2^*_s}_{2^*_s}\\
&\ge\left(\frac{1}{2\gamma}-\frac{1}{\sigma}\right)M(\left\|u_j\right\|^2)\left\|u_j\right\|^2+\left(\frac{1}{\sigma}
-\frac{1}{2^*_s}\right)\|u_j\|^{2^*_s}_{2^*_s},
\end{aligned}
\end{equation}
with $2\gamma<\sigma<2^*_s$. Hence, \eqref{ps1} and \eqref{mk}--\eqref{4.2} yield at once that as $j\to\infty$
\begin{equation}\label{4bis}
\begin{gathered}
c_\lambda+o(1)\ge\left(\frac{1}{2\gamma}-\frac{1}{\sigma}\right)M(\left\|u_j\right\|^2)\left\|u_j\right\|^2
\ge \sigma_\lambda \left\|u_j\right\|^2,\\
\sigma_\lambda=\left(\dfrac{1}{2\gamma}-\frac{1}{\sigma}\right)\kappa_\lambda>0.
\end{gathered}
\end{equation}
Therefore, $\left\{u_{j}\right\}_{j\in\mathbb{N}}$ is bounded in $Z$.

Now we can prove the validity of the {\it Palais--Smale} condition.
By using the boundedness of $\left\{u_{j}\right\}_{j\in\mathbb{N}}$ in $Z$
and by applying \cite[Lemma~4]{fiscella} and \cite[Theorem~4.9]{brezis}, there exists $u_\lambda\in Z$ such that,
up to a subsequence, still relabeled $\left\{u_{j}\right\}_{j\in\mathbb{N}}$, it follows that
\begin{equation}\label{4.4}
\begin{array}{ll}
u_{j}\rightharpoonup u_\lambda\text{ in $Z$ and in }L^{2^*_s}(\Omega), & \left\|u_j\right\|\to\alpha_\lambda, \\
u_{j}\to u_\lambda\text{ in }L^{q}(\Omega)\mbox{ and in } L^{2\gamma}(\Omega),\quad &\left\|u_j-u_\lambda\right\|_{2^*_s}\to\ell_\lambda,\\
u_{j}\to u_\lambda\text{ a.e.  in }\Omega, &|u_{j}|\le h \text{ a.e.  in }\Omega,
\end{array}
\end{equation}
with $h\in L^{2\gamma}(\Omega)\cap L^{q}(\Omega)$. Clearly $\alpha_\lambda>0$ since we are in the case in which $d_\lambda>0$.
Therefore, $M(\left\|u_j\right\|^2)\to M(\alpha_\lambda^2)>0$ as $j\to\infty$, by continuity and the fact that $0$ is
the unique zero of $M$ by $(M_2)$.

We first assert that
\begin{equation}\label{ass}
\lim_{\lambda\to\infty}\alpha_\lambda=0.
\end{equation}
Otherwise $\limsup_{\lambda\to\infty}\alpha_\lambda=\alpha>0$. Hence there is a sequence, say $k\to\lambda_k\uparrow\infty$ such that
$\alpha_{\lambda_k}\to\alpha$ as $k\to\infty$, and letting $k\to\infty$ we get from \eqref{4.2} and Lemma~\ref{infinito} that
$$0\ge\left(\frac{1}{2\gamma}-\frac{1}{\sigma}\right)M(\alpha^2)\alpha^2>0$$
by $(M_2)$, which is the desired contradiction and proves the assertion \eqref{ass}.
Moreover, $\|u_\lambda\|\le\lim_j\|u_j\|=\alpha_\lambda$ since $u_j\rightharpoonup u_\lambda$ in $Z$, so that \eqref{ass}
implies at once by the fractional Sobolev inequality
\begin{equation}\label{assl}
\lim_{\lambda\to\infty}\|u_\lambda\|_{2^*_s}=\lim_{\lambda\to\infty}\|u_\lambda\|=0.
\end{equation}
By  \eqref{f11}, \eqref{4.4} and the fact that $\left|u_j\right|^{2^*_s-2}u_j\rightharpoonup \left|u_\lambda\right|^{2^*_s-2}u_\lambda$
in $L^{2^{*'}_s}(\Omega)$, where $2^{*'}_s=2n/(n+2s)$ is the H\"older conjugate of $2^*_s$, we have
$$M(\alpha_\lambda^2)\left\langle u_\lambda,\varphi\right\rangle=\lambda\int_\Omega f(x, u_\lambda(x))\varphi(x)dx
+\int_\Omega \left|u_\lambda(x)\right|^{2^*_s-2}u_\lambda(x)\varphi(x)dx$$
for any $\varphi\in Z$.  Hence, $u_\lambda$ is a critical point of the $C^1(Z)$ functional
\begin{equation}\label{J}\mathcal J_{\alpha_\lambda}(u)=\frac12M(\alpha^2_\lambda)\|u\|^2-\lambda\int_\Omega F(x, u(x))dx-\frac1{2^*_s}\|u\|^{2^*_s}_{2^*_s}.\end{equation}
In particular, \eqref{ps1} and \eqref{4.4} imply that as $j\to\infty$
\begin{align}\label{noi}
o(1)&=\langle \mathcal J'_{\lambda}(u_j)- \mathcal J'_{\alpha_\lambda}(u_\lambda), u_j-u_\lambda\rangle= M(\|u_j\|^2)\|u_j\|^2 + M(\alpha_\lambda^2)\|u_\lambda\|^2\nonumber\\
&-  \langle u_j,u_\lambda\rangle\left[M(\|u_j\|^2)+M(\alpha_\lambda^2)\right]-\lambda\!\int_{\Omega}\!\!\big[f(x,u_j)-f(x,u_\lambda)\big](u_j-u_\lambda)\,dx\nonumber\\
 &-\int_{\Omega}\big(|u_j|^{2^*_s-2}u_j-|u_\lambda|^{2^*_s-2}u_\lambda\big)(u_j-u_\lambda)dx\\
&=M(\alpha^2_\lambda)\big(\alpha_\lambda^2-\|u_\lambda\|^2\big)-\|u_j\|^{2^*_s}_{2^*_s}+\|u_\lambda\|^{2^*_s}_{2^*_s}+o(1)\nonumber\\
&=M(\alpha^2_\lambda)\|u_j-u_\lambda\|^2-\|u_j-u_\lambda\|^{2^*_s}_{2^*_s}+o(1).\nonumber
\end{align}
Indeed, by \eqref{f11} and \eqref{4.4},
$$\lim_{j\to\infty}\int_{\Omega}\big[f(x,u_j)-f(x,u_\lambda)\big](u_j-u_\lambda)\,dx=0.$$
Moreover, again by \eqref{4.4} and the celebrated Brezis \& Lieb lemma, see \cite{brezislieb}, as $j\to\infty$
$$\|u_j\|^2=\|u_j-u_\lambda\|^2+\|u_\lambda\|^2+o(1),\qquad
\|u_j\|^{2^*_s}_{2^*_s}=\|u_j-u_\lambda\|^{2^*_s}_{2^*_s}+\|u_\lambda\|^{2^*_s}_{2^*_s}+o(1).
$$
Finally, we have used the fact that $\|u_j\|\to\alpha_\lambda$. Therefore, we have proved the main formula
\begin{equation}\label{I}
M(\alpha^2_\lambda)\lim_{j\to\infty}\|u_j-u_\lambda\|^2=\lim_{j\to\infty}\|u_j-u_\lambda\|^{2^*_s}_{2^*_s}.
\end{equation}
Using \eqref{4.2}, \eqref{4.4} and the Brezis \& Lieb lemma, we attain as $j\to\infty$
$$
\begin{aligned}
c_\lambda+o(1)&=\mathcal J_\lambda(u_j)-\frac1{\sigma}\mathcal J'_\lambda(u_j)(u_j)\ge\left(\frac{1}{\sigma}-\frac{1}{2^*_s}\right)\|u_j\|^{2^*_s}_{2^*_s}\\
&=\left(\frac{1}{\sigma}-\frac{1}{2^*_s}\right)\big\{\ell_\lambda^{2^*_s}+\|u_\lambda\|^{2^*_s}_{2^*_s}\big\}+o(1).
\end{aligned}$$
Thus, by Lemma~\ref{infinito} and \eqref{assl} also
\begin{equation}\label{liml}
\lim_{\lambda\to\infty}\ell_\lambda=0.
\end{equation}
Denote by $S$ the main fractional Sobolev constant, that is
\begin{equation}\label{S}S=\inf_{\substack{v\in H^s(\mathbb R^n)\\v\not=0}}\frac{\|v\|^2}{\|v\|^2_{2^*_s}}.\end{equation}
By \eqref{I} and the notation in \eqref{4.4},  for all $\lambda\in\mathbb R^+$
\begin{equation}\label{ll}\ell_\lambda^{2^*_s}\ge S\, M(\alpha^2_\lambda)\ell_\lambda^2.\end{equation}
We assert that $\ell_\lambda=0$ for all $\lambda\ge\lambda^*$. Otherwise there exists a sequence $k\mapsto\lambda_k\uparrow\infty$ such that
$\ell_{\lambda_k}=\ell_k>0$. Noting that \eqref{noi} implies in particular that
$$M(\alpha^2_\lambda)\big(\alpha_\lambda^2-\|u_\lambda\|^2\big)=\ell_\lambda^{2^*_s},$$
we get along this sequence, using \eqref{ll}  and denoting $\alpha_{\lambda_k}=\alpha_k$, $u_{\lambda_k}=u_k$, that
$$\big(\ell_k^{2^*_s}\big)^{2s/n}=M(\alpha^2_k)^{2s/n}\big(\alpha_k^2-\|u_k\|^2\big)^{2s/n}\ge S\, M(\alpha^2_k).$$
Hence,  we obtain for all $k$ sufficiently large by $(M_3)$ and \eqref{ass}
$$\alpha_k^{4s/n}\ge\big(\alpha_k^2-\|u_k\|^2\big)^{2s/n}\ge S\, M(\alpha^2_k)^{1-2s/n}\ge c\,S\,\alpha_k^{2(1-2s/n)},$$
where $c=a^{1-2s/n}$. Therefore, being $\alpha_k>0$ for all $k$, it follows that for all $k$ sufficiently large
$$\alpha_k^{2(4s/n-1)}\ge  c\,S.$$
This is impossible by \eqref{ass}, since $4s>n$ by assumption.
In conclusion the assertion is proved.

Hence, for all $\lambda\ge\lambda^*$
$$\lim_{j\to\infty}\|u_j-u_\lambda\|^{2^*_s}_{2^*_s}=0.$$
Thus,  $u_j\to u_\lambda$ in $Z$ as $j\to\infty$ for all $\lambda\ge\lambda^*$ by \eqref{I}, being $M(\alpha^2_{\lambda})>0$ by $(M_2)$ and the fact that $d_\lambda>0$.
This completes the proof of the first case.\medskip

\noindent {\it Case $\displaystyle\inf_{j\in\mathbb{N}}\left\|u_j\right\|=0$.} Here, either 0 is an accumulation point for the real sequence $\left\{\left\|u_{j}\right\|\right\}_{j\in\mathbb{N}}$
and so there is a subsequence of $\left\{u_{j}\right\}_{j\in\mathbb{N}}$ strongly converging to $u=0$,
or 0 is an isolated point of $\left\{\left\|u_{j}\right\|\right\}_{j\in\mathbb{N}}$, and so there is a subsequence,
denoted by $\left\{\left\|u_{j_k}\right\|\right\}_{k\in\mathbb{N}}$, such that $\displaystyle\inf_{k\in\mathbb{N}}\left\|u_{j_k}\right\|=d_\lambda>0$.
In the first case we are done, while in the latter we can proceed as before. This completes the proof of the second case and of the lemma.
\end{proof}

\begin{proof}[\bf Proof of Theorem \ref{Th1}] Lemmas \ref{mp1}, \ref{mp2} and~\ref{palais} guarantee that for any $\lambda\ge\lambda^*$ the functional $\mathcal J_\lambda$ satisfies
all the assumptions of the Mountain Pass theorem. Hence, for any $\lambda\ge\lambda^*$ there exists a critical point $u_\lambda\in Z$ for
$\mathcal J_\lambda$ at level $c_\lambda$. Since $J_\lambda(u_\lambda)=c_\lambda>0=J_\lambda(0)$ we have that $u_\lambda\not\equiv 0$.
Finally, the asymptotic behavior \eqref{asym} holds thanks to \eqref{assl}, that is \eqref{asym} is a consequence of~\eqref{4bis} and Lemma~\ref{infinito}.
\end{proof}

\section{The non--degenerate case}\label{sec non-degenerate}

In this section we study problem \eqref{P} when the Kirchhoff function $M$ has a {\it non--degenerate} nature and satisfies \eqref{m'2}.
The first objective is to extend the result given in \cite[Theorem~1]{FV}, by trying to adapt the approach used in Theorem~\ref{Th1}.
As in~\cite{FV}, in the {\it non--degenerate} case we can still cover the model $M$ given in \eqref{prot} with a general exponent $\gamma\geq1$.
However, as pointed out in the Introduction, because of the presence of a critical term in \eqref{P}, the technique used in Theorem~\ref{Th1}
does not work for any $\gamma$ in Theorem~\ref{Th2}. This fact forces in Theorem~\ref{Th2}--$(i)$ the request $1\le\gamma<\sigma/2$,  and in Theorem~\ref{Th2}--$(ii)$ 
the use of a truncation argument seems to be useful as in~\cite[Theorem~1]{FV}.\medskip

\begin{proof}[\bf Proof of Theorem \ref{Th2}]
First let us say that by $(F_1)$ and $(F_2)$ for any $\varepsilon>0$ there exists
$\delta_\varepsilon>0$ such that $f(x,t)\leq 2\varepsilon t +q\delta_\varepsilon t^{q-1}$ for all $(x,t)\in\Omega\times\mathbb R^+_0$.
Clearly in both cases of the theorem, Lemma~\ref{mp1} continues to hold, replacing \eqref{f11} by the above relation and~\eqref{m3} by~\eqref{m30}.
We now divide the proof.\smallskip

\noindent
{\it Case $(i)$}.  
Lemmas~\ref{mp2} and~\ref{infinito} can be proved in an unchanged way as in Section~\ref{sec degenerate}. 
Hence all Lemmas~\ref{mp1}--\ref{infinito} are available and it remains to prove the main Lemma~\ref{palais}.

 The proof of Lemma~\ref{palais} simplifies, but we repeat the main argument where necessary. Fix $\lambda>0$ and let $\left\{u_{j}\right\}_{j\in\mathbb{N}}\subset Z$ be a Palais--Smale sequence for $\mathcal J_\lambda$ at level $c_\lambda$. We can proceed exactly as in the proof of Lemma~\ref{palais}, with $(M_1)$ replaced by the inequality in $(i)$. Hence, \eqref{m'2}, $(F_3)$ and \eqref{ps1} yield now that as $j\to\infty$
\begin{equation}\label{ut}\begin{aligned}c_\lambda+o(1)&\ge\left(\frac{1}{2\gamma}-\frac{1}{\sigma}\right)M(\left\|u_j\right\|^2)\left\|u_j\right\|^2+\left(\frac{1}{\sigma}
-\frac{1}{2^*_s}\right)\|u_j\|^{2^*_s}_{2^*_s}\\
&\ge\left(\frac{1}{2\gamma}-\frac{1}{\sigma}\right) a \left\|u_j\right\|^2,\quad\mbox{with }\left(\frac{1}{2\gamma}-\frac{1}{\sigma}\right) a>0\end{aligned}\end{equation}
by assumptions $(i)$ and \eqref{m'2}. Therefore, $\left\{u_{j}\right\}_{j\in\mathbb{N}}$ is bounded in $Z$ and so again, up to a subsequence,
still relabeled $\left\{u_{j}\right\}_{j\in\mathbb{N}}$,
we continue to have \eqref{4.4}, where now $\alpha_\lambda\ge0$, but again $M(\left\|u_j\right\|^2)\to M(\alpha_\lambda^2)\ge a>0$ as $j\to\infty$ by~\eqref{m'2}.
Therefore,  the weak limit $u_\lambda$ is a critical point of the $C^1(Z)$ functional $\mathcal J_{\alpha_\lambda}$ defined in \eqref{J}.
Clearly, \eqref{ass} and \eqref{assl} are still valid, so that the main formulas \eqref{noi}--\eqref{liml} can be derived exactly in the same way.
Hence, thanks to \eqref{m'2} and \eqref{S} the inequality \eqref{ll} reduces to
$$\ell_\lambda^{2^*_s}\ge a\,S\, \ell_\lambda^2,$$
which, together with \eqref{liml}, yields at once that there exists $\lambda^*>0$ such that $\ell_\lambda=0$ for all $\lambda\ge\lambda^*$.

In conclusion, the Palais--Smale condition holds and we are done as in the proof of Theorem~\ref{Th1}. 

Finally,  \eqref{asym} is just a consequence of~\eqref{assl}, which is implied by~\eqref{ut}
and  Lemma~\ref{infinito}.\smallskip

\noindent{\it Case $(ii)$}. As in \cite{FV}, before solving \eqref{P} we need a truncation argument.
Take $\alpha\in\mathbb{R}$, with $0<a\le M(0)<\alpha<\sigma a/2$, which is possible being $2M(0)<\sigma a$ by assumption. Put for all $t\in\mathbb R^+_0$
$$\begin{gathered}M_\alpha(t)=\begin{cases}
M(t), & \mbox{if } M(t)\leq \alpha,\\
\alpha, & \mbox{if } M(t)>\alpha,
\end{cases}\quad\mbox{so that}\\
M_\alpha(0)=M(0),\qquad\min_{t\in\mathbb R^+_0}M_\alpha(t)=a,\end{gathered}$$
and  denote by $\mathscr{M}_\alpha$ its primitive. Let us consider the auxiliary problem
\begin{equation}\label{Pa}
\begin{cases}
-M_\alpha(\left\|u\right\|^2)\mathcal L_Ku=\lambda f(x,u)+\left|u\right|^{2^* -2}u, & \mbox{in } \Omega,\\
u=0, & \mbox{in } \mathbb{R}^{n}\setminus\Omega.\end{cases}
\end{equation}
We are going to solve \eqref{Pa}, using a mountain pass argument as done in 
{\em Case~$(i)$}, but replacing the  Kirchhoff function $M$ with $M_\alpha$.

Clearly \eqref{Pa} can be thought as the Euler--Lagrange equation of the $C^1$ functional
$$\mathcal J_{\alpha,\lambda}(u)=\frac{1}{2}\mathscr{M}_\alpha(\left\|u\right\|^2)-\lambda\int_\Omega F(x, u(x))dx
-\frac{1}{2^*_s}\|u\|^{2^*_s}_{2^*_s},\quad u\in Z,$$
as introduced in \cite{FV}. First let us observe that for the functional $\mathcal J_{\alpha,\lambda}$ all the Lemmas~\ref{mp1}--\ref{infinito}
continue to hold. Indeed, for Lemma~\ref{mp2} it is enough only to observe that \eqref{m4} is now replaced by $\mathscr M(t)\le \alpha t$ for all $t\in\mathbb R^+_0$. Similarly, also Lemma~\ref{infinito} can be proved in a simpler way, by observing that now, being $t_\lambda >0$
for all $\lambda>0$, then \eqref{3.1} becomes
$$\alpha\, t^2_\lambda \|e\|^2\ge t^2_\lambda \|e\|^2 M_\alpha(t^{2}_{\lambda}\|e\|^2)\ge  t^{2^*_s}_{\lambda}\|e\|^{2^*_s}_{2^*_s} \quad \mbox{for any }\lambda\in\mathbb R^+.
$$
This implies at once that $\{t_\lambda\}_{\lambda\in\mathbb R^+}$ is bounded in $\mathbb R$. The rest of the proof is unchanged.
Hence  all Lemmas~\ref{mp1}--\ref{infinito} are valid for $\mathcal J_{\alpha,\lambda}$ and it remains to prove for $\mathcal J_{\alpha,\lambda}$ the main Lemma~\ref{palais}.

Proceeding as in {\em Case} $(i)$, by \eqref{m'2} and~$(F_3)$ now~\eqref{ut} becomes
\begin{equation}\label{utbis}
\begin{gathered}
c_\lambda+o(1)\ge\left(\frac{a}{2}-\frac{\alpha}{\sigma}\right)\|u_j\|^2+\left(\frac{1}{\sigma}
-\frac{1}{2^*_s}\right)\|u_j\|^{2^*_s}_{2^*_s},\\
\mbox{with }\frac{a}{2}-\frac{\alpha}{\sigma}>0,
\end{gathered}
\end{equation}
being $\alpha<\sigma a/2$. While the other key formulas hold true with no relevant modifications.
Thus, arguing as in the proof of Theorem~\ref{Th1}, we can find that for all $\alpha\in(M(0),\sigma a/2)$ there exists a suitable $\lambda_0=\lambda_0(\alpha)>0$ such that problem \eqref{Pa} admits a non--trivial weak solution $u_\lambda\in Z$, with $\mathcal J_{\alpha,\lambda}(u_\lambda)=c_\lambda$. Hence, \eqref{utbis} implies  that for all $\lambda\geq\lambda_0$
$$c_\lambda\ge\left(\frac{a}{2}-\frac{\alpha}{\sigma}\right)\|u_\lambda\|^2,\quad\mbox{with }\frac{a}{2}-\frac{\alpha}{\sigma}>0,$$
so that \eqref{asym} follows at once by Lemma~\ref{infinito}.

Fix $\alpha\in(M(0),\sigma a/2)$. By \eqref{asym}
$$a\le M(0)=M_\alpha(0)=\lim_{\substack{\lambda\to\infty\\ \lambda\ge\lambda_0}}M_\alpha(\left\|u_{\lambda}\right\|^2).$$
Therefore, there exists $\lambda^*=\lambda^*(\alpha)\ge\lambda_0$ such that
$$a\le M_\alpha(\left\|u_{\lambda}\right\|^2)<\alpha\quad\mbox{for all }\lambda\geq\lambda^*.$$
In conclusion, for all $\alpha\in(M(0),\sigma a/2)$ there exists a threshold $\lambda^*=\lambda^*(\alpha)>0$ such that for all $\lambda\geq\lambda^*$ the mountain pass solution $u_\lambda$ of \eqref{Pa} is also a solution of problem~\eqref{P}.
\end{proof}

Note that with Theorem~\ref{Th2}--$(ii)$ we are able to cover the case $\gamma\geq\sigma/2$,
not allowed in Theorem~\ref{Th2}--$(i)$. However, as mentioned in the
Introduction, the approach used in the case $(ii)$ is different. Indeed, in Theorem~\ref{Th2}--$(ii)$ the Kirchhoff function $M$ could increase  faster than the other terms
of $\mathcal J_\lambda$. This makes the argument performed in Theorem~\ref{Th1} no longer applicable, and a truncation technique
is used, as in~\cite{Fig,FV}, but without any monotonicity assumption on $M$ and without assuming that $M(0)=a$.\medskip

\noindent {\bf Acknowledgement.} The authors are members of the {\em Gruppo Nazionale per l'Analisi Matematica, la Probabilit\`a e
le loro Applicazioni} (GNAMPA) of the {\em Istituto Nazionale di Alta Matematica ``G. Severi"} (INdAM).
The third author was partially supported by the MIUR Project (201274FYK7)
{\em Aspetti variazionali e perturbativi nei problemi differenziali nonlineari}.

\end{document}